\documentclass[a4paper,11pt]{article}
\usepackage{amsmath,amsthm,amssymb}
\usepackage[mathscr]{eucal}
\usepackage[bookmarks=false]{hyperref}

\numberwithin{equation}{section}

{\theoremstyle{definition}\newtheorem{definition}{Definition}

}

\newtheorem{lemma}[definition]{Lemma}
\newtheorem{theorem}[definition]{Theorem}

\newcommand{\F}{\mathbb{F}}
\newcommand{\Z}{\mathbb{Z}}
\newcommand{\C}{\mathbb{C}}
\newcommand{\cU}{\mathcal{U}}
\newcommand{\bigovt}{\overline{\bigotimes}}
\newcommand{\rL}{\operatorname{L}}
\newcommand{\rE}{\operatorname{E}}
\newcommand{\acts}[1]{\overset{#1}{\curvearrowright}}
\newcommand{\bigot}{\bigotimes}
\newcommand{\recht}{\rightarrow}
\newcommand{\eps}{\varepsilon}
\newcommand{\SL}{\operatorname{SL}}
\newcommand{\N}{\mathbb{N}}
\newcommand{\al}{\alpha}
\newcommand{\dis}{\displaystyle}
\newcommand{\si}{\sigma}
\newcommand{\actson}{\curvearrowright}

\begin{document}

\begin{center}
{\LARGE\bf An inner amenable group whose von Neumann algebra\vspace{0.5ex}\\ does not have property Gamma}

\bigskip

{\sc by Stefaan Vaes\footnote{Partially
    supported by ERC Starting Grant VNALG-200749, Research
    Programme G.0231.07 of the Research Foundation --
    Flanders (FWO) and K.U.Leuven BOF research grant OT/08/032.\\ Department of Mathematics;
    K.U.Leuven; Celestijnenlaan 200B; B--3001 Leuven (Belgium).
    \\ E-mail: stefaan.vaes@wis.kuleuven.be}}
\end{center}

\begin{abstract}\noindent
We construct inner amenable groups $G$ with infinite conjugacy classes and such that the associated II$_1$ factor does not have property Gamma of Murray and von Neumann. This solves a problem posed by Effros in 1975.
\end{abstract}

\section{Introduction}

In order to exhibit two non-isomorphic II$_1$ factors, Murray and von Neumann introduced in \cite{MvN43} their qualitative {\it property Gamma.} They showed that the group von Neumann algebra $\rL \F_n$ of the free group $\F_n$, $n \geq 2$, does not have property Gamma, while the group von Neumann algebra $\rL S_\infty$ of the group of finite permutations of $\N$ has property Gamma.

Following \cite{MvN43}, a II$_1$ factor $M$ with trace $\tau$ has property Gamma if there exists a sequence of unitaries $x_n$ in $M$ satisfying $\tau(x_n) = 0$ for all $n$ and $\|x_n y - y x_n \|_2 \recht 0$ for all $y \in M$. Here $\| \cdot \|_2$ denotes the $\rL^2$-norm on $M$ given by $\|x\|_2 = \sqrt{\tau(x x^*)}$.

In \cite{effros}, Effros aims to express property Gamma for a group von Neumann algebra $\rL G$ in terms of a group theoretic property. In this respect, he introduced the notion of \emph{inner amenability} for a countable group $G$, by requiring the existence of a mean on $G \setminus \{e\}$ that is invariant under all inner automorphisms. More precisely, $G$ is inner amenable if there exists a finitely additive measure $m$ on the subsets of $G \setminus \{e\}$, with total mass $1$ and satisfying $m(g X g^{-1}) = m(X)$ for all $X \subset G \setminus \{e\}$ and all $g \in G$. In \cite{effros}, Effros proved that if $\rL G$ is a II$_1$ factor with property Gamma, then $G$ is inner amenable. He posed the question whether the converse holds: does $\rL G$ have property Gamma whenever $G$ is an inner amenable group with infinite conjugacy classes (icc). This problem attracted a lot of attention over the years, see e.g.\ \cite[Problem 2]{harpe} and the survey \cite{bedos-harpe}.
In attempts to answer Effros' question, several groups were first shown to be inner amenable (e.g.\ Thompson's group \cite{Jol97}, Baumslag-Solitar groups \cite{stalder}), but later shown to satisfy property Gamma as well (e.g.\ \cite{Jol00}).

We solve Effros' question in the negative by providing concrete examples of inner amenable icc groups $G$ such that $\rL G$ does not have property Gamma.

The construction in this paper is inspired by the following similar phenomenology in ergodic theory of group actions, exhibited by Schmidt \cite[Example 2.7]{schmidt}. Let $G \actson (X,\mu)$ be a measure preserving action of a countable group $G$ on a standard non-atomic probability space $(X,\mu)$. The action is called \emph{strongly ergodic} if the following implication holds: whenever $\cU_n \subset X$ is a sequence of almost invariant measurable subsets (i.e.\ $\mu(g \cdot \cU_n \bigtriangleup \cU_n) \recht 0$ for all $g \in G$), then $\mu(\cU_n) (1-\mu(\cU_n)) \recht 0$. It is easy to see (e.g.\ \cite[Proposition 6.3.2]{BHV}) that whenever $G \actson (X,\mu)$ is such that the Koopman representation $G \recht \cU(\rL^2(X) \ominus \C 1)$ does not weakly contain the trivial representation, then $G \actson (X,\mu)$ is strongly ergodic. In \cite[Example 2.7]{schmidt}, Schmidt shows that the converse can fail.

\subsubsection*{Acknowledgment}
It is my pleasure to thank Sorin Popa for bringing to my attention Schmidt's example \cite[Example 2.7]{schmidt} which triggered the construction presented in this paper.

\section{Construction of the group $G$}

Fix a sequence of distinct prime numbers $p_n$. We define as follows a countable group $G$. Define
$$H_n := \Bigl(\frac{\Z}{p_n \Z}\Bigr)^3 \quad\text{and}\quad K := \bigoplus_{n=0}^\infty H_n \; .$$
Put $\Lambda = \SL(3,\Z)$, which acts on $H_n$ by automorphisms in the natural way. We denote this action as $g \cdot x$ whenever $g \in \Lambda$, $x \in H_n$. We let $\Lambda$ act on $K$ diagonally~: $(g \cdot x)_n = g \cdot x_n$ for all $g \in K$ and $n \in \N$. For every $N \in \N$, define the subgroup $K_N < K$ as
$$K_N := \bigoplus_{n = N}^\infty H_n \; .$$
We put $G_0 = K \rtimes \Lambda$ and inductively define $G_{N+1}$ as the following amalgamated free product.
$$G_N \hookrightarrow G_{N+1} := G_N *_{K_N} (K_N \times \Z) \; .$$
Note here that we view $K_N$ as a subgroup of $G_N$ by considering $K_N < K < G_0 < G_N$.
We finally define $G$ as the inductive limit of the increasing sequence of groups $G_0 \subset G_1 \subset \cdots$.

\begin{theorem}\label{thm.main}
The group $G$ is inner amenable and has infinite conjugacy classes while the II$_1$ factor $\rL G$ does not have property Gamma.
\end{theorem}

\section{Proof of Theorem \ref{thm.main}}

We denote by $\rL G$ the group von Neumann algebra of a countable group $G$, generated by the unitaries $(u_g)_{g \in G}$. We denote by $(\delta_g)_{g \in G}$ the canonical orthonormal basis of $\ell^2(G)$. Then, $\ell^2(G)$ is an $\rL G$ - $\rL G$ - bimodule, given by $u_g \delta_k u_h = \delta_{gkh}$. On $\rL G$, we consider the usual trace given by $\tau(x) = \langle \delta_e,x \delta_e \rangle$.

\begin{lemma}\label{easy-lemma}
For every $g \in G \setminus K$, the set $\{hgh^{-1} \mid h \in \Lambda\}$ is infinite. Also, $G$ has infinite conjugacy classes.
\end{lemma}
\begin{proof}
If $g \in G \setminus G_0$, take $N \geq 0$ such that $g \in G_{N+1} \setminus G_N$. From the description of $G_{N+1}$ as the amalgamated free product $G_{N+1} = G_N *_{K_N} (K_N \times \Z)$, it follows that the elements $h g h^{-1}$, $h \in G_N$, are all distinct. In particular, $\{h gh^{-1} \mid h \in \Lambda\}$ is infinite. If $g \in G_0 \setminus K$, the set $\{h gh^{-1} \mid h \in \Lambda\}$ is infinite because $\Lambda$ has infinite conjugacy classes.

Finally, assume that $g \neq e$ has a finite conjugacy class. By the first part of the proof, $g \in K$. Taking $N$ large enough, $g \in K \setminus K_N$. So, $g \in G_N \setminus K_N$ and we arrive at the contradiction that $g$ has a finite conjugacy class in $G_N *_{K_N} (K_N \times \Z)$.
\end{proof}

Denote by $(A_n,\tau)$ the tracial von Neumann algebra with $A_n \cong \C^2$ and with minimal projections $e_n,1-e_n$ such that $\tau(e_n) = p_n^{-3}$.

\begin{lemma} \label{lemma}
Define $(A,\tau) := \bigovt_{n=1}^\infty (A_n,\tau)$. There is a unique trace preserving bijective isomorphism
$$\al : A \recht \rL G \cap (\rL \Lambda)'$$
satisfying $\dis \al(e_n) = p_n^{-3} \sum_{h \in H_n} u_{h_n} \; .$
\end{lemma}
\begin{proof}
By Lemma \ref{easy-lemma}, $\rL G \cap (\rL \Lambda)' = \rL K \cap (\rL \Lambda)'$. Put $B_n = \ell^\infty(H_n)$ and define the trace $\tau$ on $B_n$ given by the normalized counting measure. View $A_n \subset B_n$ in a trace preserving way and such that $e_n$ corresponds to the function $\chi_{\{0\}}$.

Define $\Lambda \acts{\theta} B_n$ by $(\theta_g(F))(x) = F(g^{-1} \cdot x)$ for all $g \in \Lambda$, $x \in H_n$ and $F \in B_n$. Define $\Lambda \acts{\sigma} \rL K$ by $\sigma_g(u_x) = u_{g \cdot x}$ for all $g \in \Lambda$, $x \in K$. We have $H_n \cong \widehat{H_n}$ and the Fourier transform yields a trace preserving isomorphism $\al_n : B_n \recht \rL H_n$ satisfying $\al_n \circ \theta_g = \si_{(g^{-1})^T} \circ \al_n$. Here, $g^T$ denotes the transpose of $g \in \Lambda = \SL(3,\Z)$.

Put $(B,\tau) = \bigovt_{n=1}^\infty (B_n,\tau)$ and define $\Lambda \acts{\theta} B$ diagonally. The isomorphisms $\al_n$ combine into a trace preserving isomorphism $\al : B \recht \rL K$ satisfying $\al \circ \theta_g = \si_{(g^{-1})^T} \circ \al$ for all $g \in \Lambda$. We view $A$ as a von Neumann subalgebra of $B$. The lemma then follows once we have shown that $B^\Lambda = A$, where, by definition $B^\Lambda = \{b \in B \mid \theta_g(b) = b \;\;\text{for all}\;\; g \in \Lambda \}$.

The orbits of the diagonal action $\Lambda \actson (H_1 \times \cdots \times H_N)$ are precisely the sets $\cU_1 \times \cdots \times \cU_N$ where every $\cU_i$ is either $\{0\}$ or $H_i \setminus \{0\}$. Hence, $\Bigl(\bigot_{n=1}^N B_n\Bigr)^\Lambda = \bigot_{n=1}^N A_n$. The lemma follows by letting $N \recht \infty$.
\end{proof}

\begin{proof}[Proof of Theorem \ref{thm.main}]
We saw in Lemma \ref{easy-lemma} that $G$ has infinite conjugacy classes. 

Embed $\rL G \hookrightarrow \ell^2(G)$ by $x \mapsto x \delta_e$. Define $\xi_n = p_n^{3/2} \al(e_n) \delta_e$. Then, $\xi_n$ is a sequence of unit vectors in $\ell^2(G)$ satisfying $\langle \delta_e,\xi_n \rangle = p_n^{-3/2} \recht 0$. Moreover, by construction, $u_g \xi_n u_g^* = \xi_n$ whenever $g \in G_N$ and $n > N$. Hence, for every $g \in G$, the sequence $\|u_g \xi_n u_g^* - \xi_n \|_2$ is eventually zero. It follows that the adjoint representation of $G$ on $\ell^2(G) \ominus \C \delta_e$ weakly contains the trivial representation. Hence, $G$ is inner amenable (see e.g.\ \cite[Th\'{e}or\`{e}me 1]{bedos-harpe}).

Suppose that $x_n$ is a sequence of unitaries in $\rL G$, such that $\|x_n y - y x_n \|_2 \recht 0$ for all $y \in \rL G$. We have to prove that $\|x_n - \tau(x_n) 1 \|_2 \recht 0$. Denote by $\pi : G \recht \cU(\ell^2(G)) : \pi(g) \xi = u_g \xi u_g^*$ the adjoint representation. Then $\xi_n := x_n \delta_e$ is a sequence of almost $\pi$-invariant unit vectors. Denote by $P$ the orthogonal projection of $\ell^2(G)$ onto the closed subspace of $\pi(\Lambda)$-invariant vectors. Since $\Lambda$ has property (T), it follows that $\|\xi_n - P(\xi_n)\|_2 \recht 0$. This means that $\|x_n - y_n \|_2 \recht 0$, where $y_n := \rE_{\rL G \cap (\rL \Lambda)'}(x_n)$.

As we have seen in the proof of Lemma \ref{lemma}, $\rL G \cap (\rL \Lambda)' = \rL K \cap (\rL \Lambda)'$. In particular, $y_n$ belongs to the unit ball of $\rL K$. Recall that we inductively defined $G_{N+1} = G_N *_{K_N} (K_N \times \Z)$. Denote by $g_{N+1}$ the canonical generator of the copy of $\Z$ appearing in this definition of $G_{N+1}$. Using the fact that $u_{g_{N+1}}$ commutes with $\rL K_N$, we get that
\begin{align*}
u_{g_{N+1}} y_n u_{g_{N+1}}^* - y_n &= u_{g_{N+1}} \bigl(y_n - \rE_{\rL K_N}(y_n)\bigr) u_{g_{N+1}}^* +  u_{g_{N+1}} \rE_{\rL K_N}(y_n) u_{g_{N+1}}^* - y_n \\ &= u_{g_{N+1}} \bigl(y_n - \rE_{\rL K_N}(y_n)\bigr) u_{g_{N+1}}^* + \rE_{\rL K_N}(y_n) - y_n \; .
\end{align*}
Because the sets $g_{N+1} (K \setminus K_N) g_{N+1}^{-1}$ and $K$ are disjoint, the elements $u_{g_{N+1}} \bigl(y_n - \rE_{\rL K_N}(y_n)\bigr) u_{g_{N+1}}^*$ and $\rE_{\rL K_N}(y_n) - y_n$ are orthogonal. Hence,
$$\|u_{g_{N+1}} y_n u_{g_{N+1}}^* - y_n\|_2 \geq \| u_{g_{N+1}} \bigl(y_n - \rE_{\rL K_N}(y_n)\bigr) u_{g_{N+1}}^* \|_2 = \|y_n - \rE_{\rL K_N}(y_n)\|_2 \; .$$
So, for every $N$, we get that $\|y_n - \rE_{\rL K_N}(y_n) \|_2 \recht 0$ as $n \recht \infty$.

Fix $N$. Since $y_n$ commutes with $\rL \Lambda$, also $\rE_{\rL K_N}(y_n)$ commutes with $\rL \Lambda$. By Lemma \ref{lemma}, take a sequence $a_n$ in the unit ball of $\bigovt_{k=N}^\infty (A_k,\tau)$ such that $\rE_{\rL K_N}(y_n) = \al(a_n)$. Since the sequence $p_n^{-3}$ is summable, the product of the projections $1-e_n$, $n \geq N$, converges to a minimal projection $f_N$ in $\bigovt_{k=N}^\infty (A_k,\tau)$ with
$$\tau(f_N) = \prod_{n=N}^\infty (1- p_n^{-3}) \; .$$
Put $\eps_N = 1-\tau(f_N)$.
An arbitrary $a$ in the unit ball of $\bigovt_{k=N}^\infty (A_k,\tau)$ then satisfies
$$\|a - \tau(a) 1\|_2 \leq 4 \sqrt{\eps_N} \; .$$
Since $\tau(a_n) = \tau(\rE_{\rL K_N}(y_n)) = \tau(y_n)$, it follows that for all $N,n$, we have
$$\|\rE_{\rL K_N} (y_n) - \tau(y_n) 1 \|_2 \leq 4 \sqrt{\eps_N} \; .$$
Since $\eps_N \recht 0$ when $N \recht \infty$ and since for every fixed $N$, we have $\|y_n - \rE_{\rL K_N}(y_n)\|_2 \recht \infty$ when $n \recht \infty$, we conclude that $\|y_n - \tau(y_n)1\|_2 \recht 0$. So, since $\|x_n-y_n\|_2 \recht 0$, also $\|x_n - \tau(x_n)1\|_2 \recht 0$.
\end{proof}

\end{document}